\documentclass[12pt]{amsart}

\usepackage{geometry}                
\usepackage{graphicx}  
\usepackage{subfig}
\usepackage{amssymb}
\usepackage{color}
\newtheorem{theorem}{Theorem}[section]    
\newtheorem{proposition}[theorem]{Proposition}

\newtheorem{corollary}[theorem]{Corollary} 

\theoremstyle{definition}

\newtheorem{remark-number}[theorem]{Remark}

\newtheorem{remark}{Remark}             
\numberwithin{equation}{section}

\newcommand{\R}{\mathbb{R}}
\newcommand{\Q}{\mathbb{Q}}
\newcommand{\Z}{\mathbb{Z}}

\newcommand{\Mod}{{\rm Mod}}

\title{On a question of Etnyre and Van Horn-Morris}
\author{Tetsuya Ito}
\address{Department of Mathematics, Graduate School of Science, Osaka University \\ 1-1 Machikaneyama Toyonaka, Osaka 560-0043, JAPAN}
\email{tetito@math.sci.osaka-u.ac.jp}
\urladdr{http://www.math.sci.osaka-u.ac.jp/~tetito/}
\author{Keiko Kawamuro}
\address{Department of Mathematics,   
The University of Iowa, Iowa City, IA 52242, USA}
\email{keiko-kawamuro@uiowa.edu}
\date{\today}

\begin{document}

\begin{abstract}

We answer Question 6.12 in \cite{EV} asked by Etnyre and Van Horn-Morris.  
\end{abstract}

\maketitle

\section{Introduction}

Let $S$ be a compact oriented surface with boundary. 
Let $\Mod(S)$ denote the mapping class group of $S$, namely the group of isotopy classes of orientation preserving homeomorphisms of $S$ that fix the boundary $\partial S$ pointwise. 
Let $C$ be a boundary component of $S$ and let $$c(-,C):\Mod(S) \rightarrow \Q$$ denote the \emph{fractional Dehn twist coefficient $($FDTC}) of $\phi \in \Mod(S)$ with respect $C$.
See Honda, Kazez and Mati\'c's paper \cite{HKM} for the definition of the FDTC. 
For $r \in \R$ we define the following sets (\cite[p.344]{EV}):  
$$
FDTC_{r,C}(S):= \left\{ 
\phi \in \Mod(S) \: \middle| \: 
\phi = id_S \mbox{ or } 
c(\phi,C) \geq r 
\right\}
$$
\[ 
FDTC_{r}(S) :=
\left\{ \phi \in \Mod(S)\: \middle| \: 
\begin{array}{ll}
\phi = id_S \mbox{ or } \\
c(\phi,C) \geq r \mbox{ for all boundary components } C\subset \partial S\end{array}
\right\}
\]

The following theorem answers \cite[Question 6.12]{EV} of Etnyre and Van Horn-Morris: {\it For which $r\in \R$ the set $FDTC_{r}(S)$ forms a monoid?}

\begin{theorem}\label{thm:r}
Let $S$ be a surface that is not a pair of pants and has negative Euler characteristic. 
Let $C$ be a boundary component of $S$. 
The set $FDTC_{r, C}(S)$ $($and hence $FDTC_{r}(S))$ is a monoid if and only if $r>0$.
\end{theorem}

\begin{remark}
If $S$ is a pair of pants then $FDTC_{r,C}(S)$ is a monoid if and only if $r \geq 0$.
\end{remark}

Theorem~\ref{thm:r} states that $FDTC_0 (S)$ is not a monoid. But $FDTC_0 (S)$ contains $Veer^+(S)$, the monoid of \emph{right-veering mapping classes}.

\begin{corollary}\label{cor}
We have 
$$\bigcup_{r>0} FDTC_r(S) \subsetneq Veer^+(S) \subsetneq FDTC_0 (S).$$
\end{corollary}

\section{Basic study of quasi-morphisms}

As shown in \cite[Corollary 4.17]{IK2} the FTDC map $c(-,C):\Mod(S) \rightarrow \Q$  is not a homomorphism but a quasi-morphism if the surface $S$ has  negative Euler characteristic. 
In order to prove Theorem~\ref{thm:r} we first study general quasi-morphisms and obtain a monoid criterion (Theorem~\ref{theorem:qmonoid}).

Let $G$ be a group. 
A map $q: G \rightarrow \R$ is called a \emph{quasi-morphism} if $$D(q) := \sup_{g,h \in G} |q(gh)-q(g)-q(h) |$$ is finite. The value $D(q)$ is called the \emph{defect} of the quasi-morphism. A quasi-morphism $q$ is \emph{homogeneous} if $q(g^{n})=nq(g)$ for all $g \in G$ and $n \in \Z$. Every quasi-morphism can be modified to a homogeneous quasi-morphism by taking the limit: $$\overline{q}(g):= \lim_{n \to \infty} \frac{q(g^{n})}{n}$$ 
A typical example of homogeneous quasi-morphism is the {\em translation number} $$\tau: \widetilde{\textrm{Homeo}^{+}}(S^{1}) \rightarrow \R$$ given by: 
$$\tau(g)= \lim_{n \to \infty}\frac{g^{n}(0)}{n}= \lim_{n \to \infty}\frac{g^{n}(x)-x}{n} $$ 
Here $\widetilde{\textrm{Homeo}^{+}}(S^{1})$ is the group of orientation-preserving homeomorphisms of $\R$ that are lifts of orientation-preserving homeomorphisms of $S^{1}$.  
The limit $\tau(g)$ does not depend on the choice of $x\in\R$. 
An important property of $\tau$ we will use is that  
\begin{enumerate}
\item[$(*)$]
if $0<\tau(g)$ then $x<g(x)$ for all $x \in \R$. 
\end{enumerate}
Given a quasi-morphism $q:G \rightarrow \R$ and $r \in \R$ let $$G_{r}:=\left\{g \in G \: \middle| \: g=id_G \mbox{ or } q(g)\geq r\right\}.$$ 
It is easy to see that: 
\begin{proposition}
The set $G_{r}$ forms a monoid if $r \geq D(q)$. 
\end{proposition}

The following Theorem~\ref{theorem:qmonoid} gives another monoid criterion for $G_{r}$. 
We will later apply Theorem~\ref{theorem:qmonoid}  to the quasi-morphism $c(-,C)$ and prove Theorem~\ref{thm:r}.

\begin{theorem}
\label{theorem:qmonoid}
Let $q:G\to \R$ be a homogeneous quasi-morphism which is a pull-back of the translation number quasi-morphism $\tau$, namely, there is a homomorphism $f:G \rightarrow \widetilde{\textrm{Homeo}^{+}}(S^{1})$ such that $q= \tau \circ f$.
Then $G_{r}$ forms a monoid for $r>0$.
\end{theorem}

\begin{proof}
Let $r>0$. 
Assume to the contrary that $G_{r}$ id not a monoid.  
There exist $g,h \in G_{r}$ such that $gh \not \in G_{r}$. 
That is $q(gh) < r \leq q(g), q(h)$. 
Take an integer $n>0$ so that 
\begin{equation}\label{eq1}
q(g^n)-q((gh)^n)= n(q(g) - q(gh) )>D(q).
\end{equation}
By the definition of the defect we have 
\begin{equation}\label{eq2}
\left| q(g^{-n} (gh)^n) + q(g^n)-q((gh)^n) \right| \leq D(q).
\end{equation}
By (\ref{eq1}) and (\ref{eq2}) we get 
$$q(g^{-n}(gh)^n) \leq - q(g^{n}) + q((gh)^{n}) +  D(q) < -D(q)+D(q) =0.$$
Let $G=f(g)$ and $H=f(h)$. 
By the property $(*)$ we have $(G^{-n}(GH)^{n})(0) < 0$. 

On the other hand, since $0< r \leq q(h)= \tau(H)$ by the property $(*)$  again we have  $x < H(x)$ for all $x \in \R$. Therefore $(GH)^n (0) > G^n (0)$ and 
\[(G^{-n} (GH)^n) (0) > (G^{-n} G^n) (0) = 0, \]
which is a contradiction.
\end{proof}

\section{Proofs of Theorem~\ref{thm:r} and Corollary~\ref{cor}}


\begin{proof}[Proof of Theorem~\ref{thm:r}]

According to \cite[Theorem 4.16]{IK2}, if $\chi(S)<0$ the FDTC has 
$c(\phi,C) = (\tau \circ\Theta_{C})(\phi)$ for some homomorphism $\Theta_{C} : \Mod(S) \rightarrow \widetilde{\textrm{Homeo}}^{+} (S^{1})$. 
This fact along with Theorem~\ref{theorem:qmonoid} shows that  $FDTC_{r,C}(S)$ is a monoid if $\chi(S)<0$ and $r>0$. 
Since $FDTC_{r}(S)= \bigcap_{C\subset \partial S} FDTC_{r,C}(S)$  the set $FTDC_{r}(S)$ is also a monoid if $\chi(S)<0$ and $r>0$.

Next we show that  $FDTC_{r,C}(S)$ is not a monoid for $r \leq 0$. 
For any non-separating simple closed curve $\gamma$ and any boundary component $C'$ of $S$ we have $c(T_{\gamma}^{\pm 1},C') = 0$. Therefore we have 
$
T_\gamma^{\pm 1} \in FDTC_{0}(S) \subset FDTC_{r}(S),
$
i.e., 
\begin{equation}\label{0}
T_\gamma^{\pm 1} \in FDTC_{0,C}(S) \subset FDTC_{r,C}(S).
\end{equation}

(Case 1)
Recall that for any surface of genus $g \geq 2$ the group $\Mod(S)$ is generated by Dehn twists along non-separating curves (see p.114 of \cite{FM}). 
If $FDTC_{r,C}(S)$ were a monoid then this fact and (\ref{0}) imply  that 
$FDTC_{0,C}(S)=FDTC_{r,C}(S)=\Mod(S)$ which is clearly absurd. 
Thus $FDTC_{r,C}(S)$ is not a monoid if $g\geq 2$ and $r\leq 0$.

(Case 2)
If $g=0$ and $|\partial S| = 4$ let $a, b, c, d$ be the boundary components and $x, y, z$ be the simple closed curves as shown in Figure~\ref{fig:A}-(1). 
Let $r \leq 0$ and $C \in \{a, b, c, d\}$. 
Since $x, y, z$ are non-separating $$T_x^{\pm 1}, T_y^{\pm 1}, T_z^{\pm 1}\in FDTC_{0, C}(S) \subset FDTC_{r, C}(S).$$  
By the {\em lantern relation}, for any positive integer $n$ with $-n <r$ we have   $$c((T_x T_y T_z)^{-n}, C) = c(T_a^{-n} T_b^{-n} T_c^{-n} T_d^{-n}, C)= -n$$ thus $(T_x T_y T_z)^{-n} \notin FDTC_{r, C}(S)$. 
This shows that $FDTC_{r, C}(S)$ is not a monoid for all $r\leq 0$ and $C \in\{a, b, c, d\}$.   

\begin{figure}[htbp]
\begin{center}
\includegraphics*[width=100mm]{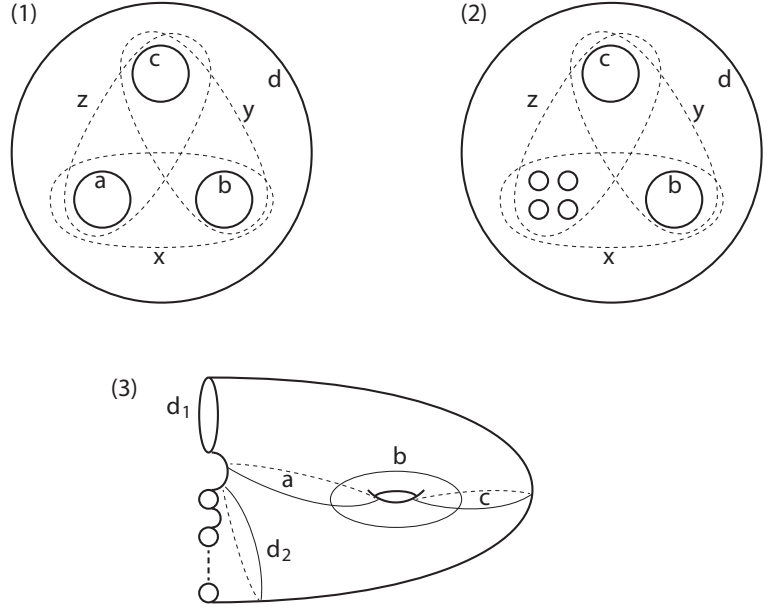}
\caption{(1) Case 2. (2) Case 3. (3) Case 5.}
\label{fig:A}
\end{center}
\end{figure}

(Case 3)
If $g=0$ and $n=|\partial S| > 4$ add $n-3$ additional boundary components $a_1,\dots, a_{n-3}$ in the place of $a$ as shown in Figure~\ref{fig:A}-(2). 
By a similar argument using the lantern relation we can show that $FDTC_{r, C}(S)$ is not a monoid for all $r\leq 0$ and any $C =b, c, d$. By the symmetry of the surface we can further show that $FDTC_{r, C}(S)$ is not a monoid for all $r\leq 0$ and $C =a_1,\dots, a_{n-3}$. 

(Case 4)
If $g=1$ and $|\partial S| =1$ 
the group $\Mod(S)$ is generated by Dehn twists about non-separating curves. Thus this case is subsumed into Case 1. 


(Case 5)
If $g=1$ and $|\partial S| \geq 2$  applying the 3-chain relation \cite[Proposition 4.12]{FM} to the curves in Figure~\ref{fig:A}-(3)
we get $c((T_a T_b T_c)^{-4n}, d_1) = c((T_{d_1})^{-n} (T_{d_2})^{-n}, d_1)= -n$. 
By the same argument as in Case 2 we can show that $FDTC_{r, d_1}(S)$ is not a monoid for all $r\leq 0$.

Parallel arguments show that $FDTC_r(S)$ does not form a monoid for $r\leq 0$.  
\end{proof}

We close the paper by proving Corollary~\ref{cor}.

\begin{proof}[Proof of Corollary~\ref{cor}] 
Let $\gamma \subset S$ be a non-separating simple closed curve.
By (\ref{0}) we observe that $$T_\gamma \in Veer^+(S) \setminus (\bigcup_{r>0} FDTC_r(S)) \mbox{ and }T_\gamma^{-1} \in FDTC_0(S) \setminus  Veer^+(S).$$
\end{proof}

\section*{Acknowledgements}
The authors thank John Etnyre for pointing out an error in an early  draft of the paper.  
TI was partially supported by JSPS Grant-in-Aid for Young Scientists (B) 15K17540.
KK was partially supported by NSF grant DMS-1206770.

\end{document}